\documentclass[12pt]{article}%
\usepackage{graphicx}
\usepackage{amsmath}
\usepackage{amsfonts}
\usepackage{amssymb}%
\providecommand{\U}[1]{\protect\rule{.1in}{.1in}}
\providecommand{\U}[1]{\protect\rule{.1in}{.1in}}
\providecommand{\U}[1]{\protect\rule{.1in}{.1in}}
\newtheorem{theorem}{Theorem}

\newtheorem{algorithm}[theorem]{Algorithm}

\newtheorem{condition}[theorem]{Condition}

\numberwithin{equation}{section}
\numberwithin{theorem}{section}
\newtheorem{definition}[theorem]{Definition}

\newtheorem{lemma}[theorem]{Lemma}

\newtheorem{problem}[theorem]{Problem}

\newtheorem{remark}[theorem]{Remark}

\newenvironment{proof}[1][Proof]{\textbf{#1.} }{\ \rule{0.5em}{0.5em}}
\begin{document}

\title{\textbf{A von Neumann Alternating Method for Finding Common Solutions to
Variational Inequalities}\bigskip}
\author{Yair Censor$^{1}$, Aviv Gibali$^{2}$ and Simeon Reich$^{2}$\bigskip\\$^{1}$Department of Mathematics, University of Haifa,\\Mt.\ Carmel, 31905 Haifa, Israel \bigskip\\$^{2}$Department of Mathematics,\\The Technion - Israel Institute of Technology\\Technion City, 32000 Haifa, Israel}
\date{\textit{Dedicated to Professor V. Lakshmikantham}\\
\textit{on the occasion of his retirement}\bigskip\\
August 16, 2011. Revised: January 16, 2012.}
\maketitle

\begin{abstract}
Modifying von Neumann's alternating projections algorithm, we obtain an
alternating method for solving the recently introduced Common Solutions to
Variational Inequalities Problem (CSVIP). For simplicity, we mainly confine
our attention to the two-set CSVIP, which entails finding common solutions to
two unrelated variational inequalities in Hilbert space.

\end{abstract}

\bigskip

\textbf{$2010$ MSC}: 47H05, 47H09, 47J20, 49J40

\section{Introduction}

A new problem, called the \textit{Common Solutions to Variational Inequalities
Problem }(CSVIP) has recently been introduced in \cite[Subsection
7.2]{cgr-svip} and further studied in \cite{cgrs10}. In \cite{cgr-svip} it was
considered as a special case of the \textit{Split Variational Inequality
Problem} (SVIP) introduced therein. The CSVIP consists of finding common
solutions to unrelated variational inequalities. In the present paper we
propose a new algorithm for solving the two-set CSVIP, which stems from the
classical von Neumann alternating projections algorithm \cite{Neumann}. We
also extend this algorithm to two methods for solving the general CSVIP, which
concerns any finite number of sets.

We first recall the general form of the CSVIP (for single-valued operators).

\begin{problem}
\label{prob:cvipp}Let $\mathcal{H}$ be a real Hilbert space. Let there be
given, for each $i=1,2,\ldots,N$, an operator $A_{i}:\mathcal{H}%
\rightarrow\mathcal{H}$ and a nonempty, closed and convex subset $K_{i}%
\subset\mathcal{H}$, with $\bigcap_{i=1}^{N}K_{i}\neq\emptyset$. The CSVIP
(for single-valued operators) is to find a point $x^{\ast}\in\bigcap_{i=1}%
^{N}K_{i}$ such that, for each $i=1,2,\ldots,N,$%
\begin{equation}
\left\langle A_{i}(x^{\ast}),x-x^{\ast}\right\rangle \geq0\text{ for all }x\in
K_{i}\text{, }i=1,2,\ldots,N. \label{eq:CSVIP}%
\end{equation}

\end{problem}

For simplicity, in this paper we mainly confine our attention to the case
where $i=2$. Denoting $A_{1}=f$, $A_{2}=g$ and the nonempty, closed and convex
subsets $K_{1}$ and $K_{2}$ by $C$ and $Q$, respectively, we get the following
two-set CSVIP.

\begin{problem}
\label{prob:csvip}Let $\mathcal{H}$ be a real Hilbert space, and let $C$ and
$Q$ be two nonempty closed and convex subsets of $\mathcal{H}$ with $C\cap
Q\neq\emptyset$. Given two operators $f$ and $g$ from $\mathcal{H}$ into
itself, the two-set CSVIP is to find a point $x^{\ast}\in C\cap Q$ such that
\begin{align}
\left\langle f(x^{\ast}),x-x^{\ast}\right\rangle  &  \geq0\text{ for all }x\in
C\label{eq:vip1}\\
&  \text{and}\nonumber\\
\left\langle g(x^{\ast}),y-x^{\ast}\right\rangle  &  \geq0\text{ for all }y\in
Q. \label{eq:vip2}%
\end{align}

\end{problem}

If we denote by $SOL(C,f)$ and $SOL(Q,g)$ the solution sets of (\ref{eq:vip1})
and (\ref{eq:vip2}), respectively, then Problem \ref{prob:csvip} is to find a
point $x^{\ast}\in SOL(C,f)\cap SOL(Q,g)$.

Looking at (\ref{eq:vip1}) separately, we get the well-known
\textit{Variational Inequality Problem} (VIP), first introduced by Hartman and
Stampacchia in 1966 (see \cite{hs}). The importance of VIPs in Nonlinear
Analysis and Optimization Theory stems from the fact that several fundamental
problems can be formulated as VIPs, e.g., constrained and unconstrained
minimization, finding solutions to systems of equations, and saddle-point
problems. See the book by Kinderlehrer and Stampacchia \cite{KS80} for a wide
range of other applications of VIPs. For an excellent treatise on variational
inequality problems in finite-dimensional spaces, see the two-volume book by
Facchinei and Pang \cite{FP03}. The books by Konnov \cite{Konnov-book} and
Patriksson \cite{Patriksson} contain extensive studies of VIPs including
applications, algorithms and numerical results.

Another motivation for defining and studying the CSVIP in \cite{cgr-svip,
cgrs10} originates in the simple observation that if we choose all $A_{i}=0$,
then the problem reduces to that of finding a point $x^{\ast}\in\bigcap
_{i=1}^{N}K_{i}$ in the nonempty intersection of a finite family of closed and
convex sets, which is the well-known \textit{Convex Feasibility Problem}
(CFP). If the sets $K_{i}$ are the fixed point sets of a family of operators
$T_{i}:\mathcal{H\rightarrow H}$, then the CFP is the \textit{Common Fixed
Point Problem} (CFPP). These problems have been intensively studied over the
past decades both theoretically (existence, uniqueness, and properties of
solutions) and algorithmically (devising iterative procedures which generate
sequences that converge, finitely or asymptotically, to a solution).

Our alternating method for solving the two-set CSVIP is inspired by von
Neumann's original alternating projections method. Von Neumann \cite{Neumann}
presented a method for calculating the orthogonal projection onto the
intersection of two closed subspaces in Hilbert space. Let $\mathcal{H}$ be a
real Hilbert space, and let $A$ and $B$ be closed subspaces. Choose
$x\in\mathcal{H}$ and construct the sequences $\left\{  a^{k}\right\}
_{k=0}^{\infty}$ and $\left\{  b^{k}\right\}  _{k=0}^{\infty}$ by%
\begin{equation}
\left\{
\begin{array}
[c]{l}%
b^{0}=x,\\
a^{k}=P_{A}(b^{k-1})\text{ and }b^{k}=P_{B}(a^{k}),\text{ }k=1,2,\ldots,
\end{array}
\right.  \label{eq:apm}%
\end{equation}
where $P_{A}$ and $P_{B}$ denote the orthogonal projection operators of
$\mathcal{H}$ onto $A$ and $B$, respectively. Von Neumann showed \cite[Lemma
22, page 475]{Neumann} that both sequences $\left\{  a^{k}\right\}
_{k=0}^{\infty}$ and $\left\{  b^{k}\right\}  _{k=0}^{\infty}$ converge
strongly to $P_{A\cap B}(x)$. This algorithm is known as von Neumann's
alternating projections method. Observe that not only the sequences converge
strongly, but also that their common limit is the nearest point to $x$ in
$A\cap B$. For recent elementary geometric proofs of von Neumann's result, see
\cite{kr1, kr2}. In 1965 Bregman \cite{Bregman} established the weak
convergence of the sequence of alternating nearest point mappings between two
closed and convex intersecting subsets of a Hilbert space. See also \cite{bbr,
br}. In 2005 Bauschke, Combettes and Reich \cite{bcr} studied the alternating
resolvents method for finding a common zero of two maximal monotone mappings
(see also the recent paper of Boikanyo and Moro\c{s}anu \cite{bm}). We propose
an alternating method which employs two averaged operators in the sense of
\cite{bbr}. In this connection, we note that not all averaged operators are
resolvents (of monotone mappings).

Our paper is organized as follows. Section \ref{sec:Preliminaries} contains
some preliminaries. In Section \ref{sec:Algorithm} we present our alternating
method for solving the two-set CSVIP and establish its convergence. In Section
\ref{sec:General-CSVIP} we extend our algorithm to two methods for solving the
general CSVIP.

\section{Preliminaries\label{sec:Preliminaries}}

Let $\mathcal{H}$ be a real Hilbert space with inner product $\langle
\cdot,\cdot\rangle$ and induced norm $\Vert\cdot\Vert,$ and let $D$ be a
nonempty,
closed and convex subset of $\mathcal{H}$. We write $x^{k}\rightharpoonup x$
to indicate that the sequence $\left\{  x^{k}\right\}  _{k=0}^{\infty}$
converges weakly to $x,$ and $x^{k}\rightarrow x$ to indicate that the
sequence $\left\{  x^{k}\right\}  _{k=0}^{\infty}$ converges strongly to $x.$

We now recall some definitions and properties of several classes of operators.

\begin{definition}
Let $h:\mathcal{H}\rightarrow\mathcal{H}$ be an operator and let
$D\subset\mathcal{H}.$

(i) The operator $h$ is called \texttt{Lipschitz continuous} on $D\subset H$
with constant $L>0$ if%
\begin{equation}
\Vert h(x)-h(y)\Vert\leq L\Vert x-y\Vert\text{\ for all\ }x,y\in D.
\end{equation}

(ii) The operator $h$ is called \texttt{nonexpansive} on $D$ if it is
$1$-Lipschitz continuous.

(iii) The operator $h$ is called \texttt{inverse strongly monotone} with
constant $\beta>0$ \texttt{(}$\beta$-ism\texttt{)} on $D$ if%
\begin{equation}
\langle h(x)-h(y),x-y\rangle\geq\beta\Vert h(x)-h(y)\Vert^{2}\text{ for all
}x,y\in D.
\end{equation}

(iv) The operator $h$ is called \texttt{firmly nonexpansive}
\texttt{\cite{Goebel}} on $D$ if%
\[
\left\langle h(x)-h(y),x-y\right\rangle \geq\left\Vert h(x)-h(y)\right\Vert
^{2}\text{ for all }x,y\in D\text{,}%
\]
in other words, $h$ is $1$-ism.

(v) The operator $h$ is called \texttt{averaged} \cite{bbr} if there exists a
nonexpansive operator $N:\mathcal{H}\rightarrow\mathcal{H}\ $and a number
$c\in(0,1)$ such that%
\begin{equation}
h=(1-c)I+cN.
\end{equation}
In this case, we say that $h$ is $c$-av \cite{byrne04}.

(vi) We say that a nonexpansive operator $h$ \texttt{satisfies Condition (W)}
\cite{dr} if whenever $\{x^{k}-y^{k}\}_{k=1}^{\infty}$ is bounded and $\Vert
x^{k}-y^{k}\Vert-\Vert h(x^{k})-h(y^{k})\Vert\rightarrow0$, it follows that
$(x^{k}-y^{k})-(h(x^{k})-h(y^{k}))\rightharpoonup0.$

(vii) The operator $h$ is called \texttt{strongly nonexpansive} \cite{br} if
it is nonexpansive and whenever $\{x^{k}-y^{k}\}_{k=1}^{\infty}$ is bounded
and $\Vert x^{k}-y^{k}\Vert-\Vert h(x^{k})-h(y^{k})\Vert\rightarrow0$, it
follows that $(x^{k}-y^{k})-(h(x^{k})-h(y^{k}))\rightarrow0$.
\end{definition}

\begin{definition}
Let $%
\mathbb{N}
$ be the set of natural numbers, $\{h_{1},h_{2},\ldots\}$ be a sequence of
operators, and $r:%
\mathbb{N}
\rightarrow%
\mathbb{N}
$. An unrestricted (or random) product of these operators is the sequence
$\{S_{n}\}_{n\in%
\mathbb{N}
}$ defined by $S_{n}=h_{r(n)}h_{r(n-1)}\cdots h_{r(1)}$.
\end{definition}

Note that inverse strong monotonicity is also known as the \texttt{Dunn
property} \cite{dunn, rc}. It is easy to see that a $\beta$-ism operator is
Lipschitz continuous with constant $1/\beta$. Some of the relations between
these classes of operators are given below. For more details, see\ Bruck and
Reich \cite{br},\ Baillon et al. \cite{bbr}, Goebel and Reich \cite{Goebel},
and Byrne \cite{byrne04}.

\begin{remark}
\label{remark:av}(i) An operator $h$ is averaged if and only if its complement
$G:=I-h$ is $\nu$-ism for some $\nu>1/2$.

(ii) The operator $h$ is firmly nonexpansive if and only if its complement
$I-h$ is firmly nonexpansive.

(iii) The operator $h$ is firmly nonexpansive if and only if it is $1/2$-averaged.

(iv) If $h_{1}$ and $h_{2}$ are $c_{1}$-av and $c_{2}$-av, respectively, then
their composition $S=h_{1}h_{2}$ is $(c_{1}+c_{2}-c_{1}c_{2})$-av.

(v) Every averaged operator is strongly nonexpansive and therefore satisfies
condition (W).
\end{remark}

Let $D$ be a closed and convex subset of $\mathcal{H}$. For every point
$x\in\mathcal{H},$\ there exists a unique nearest point in $D$, denoted by
$P_{D}(x)$. This point satisfies%
\begin{equation}
\left\Vert x-P_{D}\left(  x\right)  \right\Vert \leq\left\Vert x-y\right\Vert
\text{\textit{ }for all}\mathit{\ }y\in D.
\end{equation}
The operator $P_{D}$ is called the metric projection or the nearest point
mapping of $\mathcal{H}$ onto $D$. The metric projection $P_{D}$ is
characterized by the fact that $P_{D}\left(  x\right)  \in D$ and
\begin{equation}
\left\langle y-P_{D}\left(  x\right)  ,x-P_{D}\left(  x\right)  \right\rangle
\leq0\text{ for all }x\in\mathcal{H},\text{ }y\in D. \label{eq:Proj}%
\end{equation}
It is also well known that the operator $P_{D}$ is averaged (see, e.g.,
\cite[page 17]{Goebel}).

\begin{definition}
A sequence $\left\{  x^{k}\right\}  _{k=0}^{\infty}\subset\mathcal{H}$ is
called \texttt{Fej\'{e}r-monotone} with respect to $D$ if for every $u\in D$%
\begin{equation}
\Vert x^{k+1}-u\Vert\leq\Vert x^{k}-u\Vert\text{ for all }k\geq0.
\end{equation}

\end{definition}

Next we recall the definition of a maximal monotone mapping.

\begin{definition}
Let $M:\mathcal{H}\rightarrow2^{\mathcal{H}}\mathcal{\ }$be a set-valued
mapping defined on a real Hilbert space $\mathcal{H}$.

(i) The \texttt{resolvent} of $M$ with parameter $\lambda$ is the operator
$J_{\lambda}^{M}:=\left(  I+\lambda M\right)  ^{-1}$, where $I$ is the
identity operator.

(ii) $M$ is called a \texttt{maximal monotone mapping} if $M$ is
\texttt{monotone}, i.e.,%
\begin{equation}
\left\langle u-v,x-y\right\rangle \geq0,\text{ for all }u\in M(x)\text{ and
}v\in M(y),
\end{equation}
and the graph $G(M)$ of $M$,%
\begin{equation}
G(M):=\left\{  \left(  x,u\right)  \in H\times H\mid u\in M(x)\right\}  ,
\end{equation}
is not properly contained in the graph of any other monotone mapping.
\end{definition}

\begin{definition}
Let $C$ be a nonempty, closed and convex subset of $\mathcal{H}$. Denote by
$N_{C}\left(  v\right)  $ the \texttt{normal cone}\textit{ }of $C$ at $v\in
C$, i.e.,%
\begin{equation}
N_{C}\left(  v\right)  :=\left\{  z\in\mathcal{H}\mid\left\langle
z,y-v\right\rangle \leq0\text{ for all }y\in C\right\}  . \label{eq:nc}%
\end{equation}

\end{definition}

Consider now the variational inequality with respect to the set $D$ and the
operator $h$:%
\begin{equation}
\left\langle h(x^{\ast}),x-x^{\ast}\right\rangle \geq0\text{ for all }x\in D.
\end{equation}

Define the mapping $M$ as follows:%
\begin{equation}
M(v):=\left\{
\begin{array}
[c]{cc}%
h(v)+N_{D}\left(  v\right)  , & v\in D,\\
\emptyset, & \text{otherwise.}%
\end{array}
\right.
\end{equation}
Under a certain continuity assumption on $h$ (which every ism operator
satisfies), Rockafellar \cite[Theorem 5, p. 85]{Rockafellar} showed that $M$
is a maximal monotone mapping and $M^{-1}\left(  0\right)  =SOL(D,h)$.

\begin{remark}
\label{remark:resolvent} It is well known that for $\lambda>0$,

(i) $M$ is monotone if and only if the resolvent $J_{\lambda}^{M}$ of $M$ is
single-valued and firmly nonexpansive.

(ii) $M$ is maximal monotone if and only if $J_{\lambda}^{M}$ is
single-valued, firmly nonexpansive and its \texttt{domain} is all of
$\mathcal{H}$, where%
\begin{equation}
\operatorname*{dom}(J_{\lambda}^{M}):=\left\{  x\in\mathcal{H}\mid J_{\lambda
}^{M}(x)\neq\emptyset\right\}  .
\end{equation}

(iii)%
\begin{equation}
0\in M(x^{\ast})\Leftrightarrow x^{\ast}\in\operatorname{Fix}(J_{\lambda}%
^{M}), \label{eq:res-fix}%
\end{equation}
where $\operatorname{Fix}(J_{\lambda}^{M})$ denotes the fixed point set of
$J_{\lambda}^{M}.$
\end{remark}

It is known that the metric projection operator coincides with the resolvent
of the normal cone.

Now we recall the following lemma \cite{cgr-svip}.

\begin{lemma}
Let $\mathcal{H}$ be a real Hilbert space and let $D\subset\mathcal{H}$ be
nonempty, closed and convex. Let $h:\mathcal{H}\rightarrow\mathcal{H}$ be an
$\alpha$-ism operator. If $D\cap\{x\in\mathcal{H}\mid h(x)=0\}\neq\emptyset,$
then $x^{\ast}\in SOL(D,h)$ if and only if $h(x^{\ast})=0$.
\end{lemma}

Using the characterization of the metric projection (\ref{eq:Proj}), we get
another connection between the solution set of a variational inequality
problem and the fixed point set of a certain operator, namely, for any
$\lambda>0,$%
\begin{equation}
SOL(D,h)=\operatorname*{Fix}\left(  P_{D}(I-\lambda h)\right)  .
\label{eq:fix-vip}%
\end{equation}
Indeed,%
\begin{equation}
x^{\ast}\in\operatorname*{Fix}\left(  P_{D}\left(  I-\lambda h\right)
\right)  \Leftrightarrow P_{D}\left(  x^{\ast}-\lambda h(x^{\ast})\right)
=x^{\ast}%
\end{equation}
and by (\ref{eq:Proj})$,$ we have for all $x\in D$ and $\lambda>0,$%
\begin{align}
0  &  \leq\left\langle x^{\ast}-\lambda h(x^{\ast})-P_{D}\left(  x^{\ast
}-\lambda h(x^{\ast})\right)  ,P_{D}\left(  x^{\ast}-\lambda h(x^{\ast
})\right)  -x\right\rangle \nonumber\\
&  =\left\langle x^{\ast}-\lambda h(x^{\ast})-x^{\ast},x^{\ast}-x\right\rangle
=\left\langle -\lambda h(x^{\ast}),x^{\ast}-x\right\rangle \nonumber\\
&  =\lambda\left\langle h(x^{\ast}),x-x^{\ast}\right\rangle .
\end{align}

Next we present another useful property of the operator $P_{D}\left(
I-\lambda h\right)  $ (cf. \cite{Masad+Reich}).

\begin{lemma}
\label{Lemma:av} Let $\mathcal{H}$ be a real Hilbert space and let
$D\subset\mathcal{H}$ be nonempty, closed and convex. Let $h:\mathcal{H}%
\rightarrow\mathcal{H}$ be a $\beta$-ism operator on $\mathcal{H}$. If
$\lambda\in(0,2\beta),$ then the operator $P_{D}\left(  I-\lambda h\right)  $
is averaged.
\end{lemma}

\begin{proof}
We first prove that the operator $I-\lambda h$ is averaged. More precisely, we
claim that if $h$ is $\beta$-ism, then the operator $I-\lambda h$ is averaged
for $\lambda\in(0,2\beta)$. Indeed, take $c\in(0,1)$ such that $c\geq
\lambda/(2\beta)$ and set $N:=I-{\displaystyle}\frac{{\displaystyle}\lambda
}{{\displaystyle}c}h.$ Then $I-\lambda h=(1-c)I+cN$ and $N$ is nonexpansive:%
\begin{align}
&  \left\Vert x-y\right\Vert ^{2}-\left\Vert N(x)-N(y)\right\Vert
^{2}\nonumber\\
&  =\left\Vert x-y\right\Vert ^{2}-\left(  \left\Vert x-y\right\Vert
^{2}-2\frac{\lambda}{c}\left\langle h(x)-h(y),x-y\right\rangle +\left(
\frac{\lambda}{c}\right)  ^{2}\Vert h(x)-h(y)\Vert^{2}\right) \nonumber\\
&  =2\frac{\lambda}{c}\left\langle h(x)-h(y),x-y\right\rangle -\left(
\frac{\lambda}{c}\right)  ^{2}\Vert h(x)-h(y)\Vert^{2}\nonumber\\
&  \geq\frac{\lambda}{c}\left(  2\beta\Vert h(x)-h(y)\Vert^{2}-\frac{\lambda
}{c}\Vert h(x)-h(y)\Vert^{2}\right) \nonumber\\
&  =\frac{\lambda}{c}\left(  2\beta-\frac{\lambda}{c}\right)  \Vert
h(x)-h(y)\Vert^{2}\geq0.
\end{align}
Now, since the metric projection $P_{D}$ is averaged, so is the composition
$P_{D}\left(  I-\lambda h\right)  $ (see Remark \ref{remark:av}(iii)).
\end{proof}

While the metric projection operator is the resolvent operator of the normal
cone operator, the composed operator $P_{D}\left(  I-\lambda h\right)  $ need
not be a resolvent of a monotone mapping.

\begin{lemma}
\label{Lemma:fix}\cite{br,byrne04} If $U:\mathcal{H}\rightarrow\mathcal{H}$ an
$V:\mathcal{H}\rightarrow\mathcal{H}$ are averaged operators and
$\operatorname*{Fix}\left(  U\right)  \cap\operatorname*{Fix}\left(  V\right)
\neq\emptyset$, then $\operatorname*{Fix}\left(  U\right)  \cap
\operatorname*{Fix}\left(  V\right)  =\operatorname*{Fix}\left(  UV\right)
=\operatorname*{Fix}\left(  VU\right)  .$
\end{lemma}

The next lemma was proved by Takahashi and Toyoda \cite[Lemma 3.2]{Takahashi}.
In this connection, see also \cite[Proposition 2.1]{Reich77}.

\begin{lemma}
\label{Lemma:Takahashi} Let $\mathcal{H}$ be a real Hilbert space and let
$D\subset\mathcal{H}$ be nonempty, closed and convex. Let the sequence
$\left\{  x^{k}\right\}  _{k=0}^{\infty}\subset\mathcal{H}$ be
Fej\'{e}r-monotone with respect to $D$. Then the sequence $\left\{
P_{D}\left(  x^{k}\right)  \right\}  _{k=0}^{\infty}$ converges strongly to
some $z\in D.$
\end{lemma}

Next we recall a theorem of Opial's \cite{Opial67}, which is also known in the
literature as the Krasnosel'ski\u{\i}-Mann theorem.

\begin{theorem}
\label{The:Opial} Let $\mathcal{H}$ be a real Hilbert space and
let $D\subset
\mathcal{H}$ be nonempty, closed and convex. Assume that $h:D\rightarrow D$ is
an averaged operator with $\operatorname*{Fix}(h)\neq\emptyset$. Then, for an
arbitrary $x^{0}\in D$, the sequence $\left\{  x^{k+1}=h(x^{k})\right\}
_{k=0}^{\infty}$ converges weakly to $z\in\operatorname*{Fix}(h)$.
\end{theorem}

\begin{remark}
The convergence obtained in Theorem \ref{The:Opial} is not strong in general
\cite{gl, bmr}.
\end{remark}

\section{The Algorithm\label{sec:Algorithm}}

In this section we introduce our modified von Neumann alternating method for
solving the two-set CSVIP (\ref{eq:vip1}) and (\ref{eq:vip2}). Let
$\Gamma:=\Gamma(C,Q,f,g):=SOL(C,f)\cap SOL(Q,g).$

The following conditions are needed for our convergence theorem.

\begin{condition}
\label{Condition:a} The operators $f:\mathcal{H}\rightarrow\mathcal{H}$ and
$g:\mathcal{H}\rightarrow\mathcal{H}$ are $\alpha_{1}$-ism and $\alpha_{2}%
$-ism, respectively.
\end{condition}

\begin{condition}
\label{Condition:b} $\lambda\in(0,2\alpha),$ where $\alpha:=\min\{\alpha
_{1},\alpha_{2}\}.$
\end{condition}

\begin{condition}
\label{Condition:c} $\Gamma\neq\emptyset.$
\end{condition}

\begin{algorithm}
\label{Alg-Alternating}$\left.  {}\right.  $

\textbf{Initialization:} Select an arbitrary starting point $x^{0}%
\in\mathcal{H}$.

\textbf{Iterative step:} Given the current iterate $x^{k},$ compute%
\begin{equation}
y^{k}=\left(  P_{Q}(I-\lambda g)\right)  \left(  x^{k}\right)  \text{ and
}x^{k+1}=\left(  P_{C}(I-\lambda f)\right)  \left(  y^{k}\right)  .
\label{eq:alt2}%
\end{equation}

\end{algorithm}

Note that (\ref{eq:alt2}) is actually an alternating method, that is,%
\begin{align}
x^{k+1}  &  =\left(  P_{C}(I-\lambda f)\right)  \left(  P_{Q}(I-\lambda
g)\right)  \left(  x^{k}\right) \nonumber\\
&  =P_{C}\left(  P_{Q}\left(  x^{k}-\lambda g\left(  x^{k}\right)  \right)
-\lambda f\left(  P_{Q}\left(  x^{k}-\lambda g\left(  x^{k}\right)  \right)
\right)  \right)  . \label{eq:alt1}%
\end{align}

An illustration of the iterative step of Algorithm \ref{Alg-Alternating} is
presented in Figure \ref{Fig:1}.%
\begin{figure}
[h]
\begin{center}
\includegraphics[
natheight=5.624700in,
natwidth=7.499600in,
height=3.8311in,
width=5.3385in
]%
{Alternating.jpg}%
\caption{Illustration of the iterative step of Algorithm \ref{Alg-Alternating}%
.}%
\label{Fig:1}%
\end{center}
\end{figure}

\begin{theorem}
\label{Theorem:1} Let $\mathcal{H}$ be a real Hilbert space, and let $C$, $Q$
be two nonempty closed and convex subsets of $\mathcal{H}$. Assume that
Conditions \ref{Condition:a}-\ref{Condition:c} hold and set $\alpha
:=\min\{\alpha_{1},\alpha_{2}\}$. Then any sequence $\left\{  x^{k}\right\}
_{k=0}^{\infty}$ generated by Algorithm \ref{Alg-Alternating} converges weakly
to a point $x^{\ast}\in\Gamma,$ and furthermore,%
\begin{equation}
x^{\ast}=\lim_{k\rightarrow\infty}P_{\Gamma}(x^{k}).
\end{equation}

\end{theorem}

\begin{proof}
Let $\lambda\in(0,2\alpha).$ By Lemma \ref{Lemma:av}, the operators
$P_{C}(I-\lambda f)$ and $P_{Q}(I-\lambda g)$ are averaged and so is their
composition $\left(  P_{C}(I-\lambda f)\right)  (P_{Q}(I-\lambda g))$ (Remark
\ref{remark:av}). Since $\Gamma\neq\emptyset,$ Opial's theorem (Theorem
\ref{The:Opial}) guarantees that any sequence $\left\{  x^{k}\right\}
_{k=0}^{\infty}$ generated by Algorithm \ref{Alg-Alternating} converges weakly
to a point $x^{\ast}\in\operatorname*{Fix}\left(  \left(  P_{C}(I-\lambda
f)\right)  (P_{Q}(I-\lambda g))\right)  .$ Combining the assumption
$\Gamma\neq\emptyset$ with Lemma \ref{Lemma:fix}, we obtain%
\begin{align}
\operatorname*{Fix}(P_{C}(I-\lambda f))\cap\operatorname*{Fix}(P_{Q}(I-\lambda
g))  &  =\operatorname*{Fix}\left(  \left(  P_{C}(I-\lambda f)\right)
(P_{Q}(I-\lambda g))\right) \nonumber\\
&  =\operatorname*{Fix}\left(  (P_{Q}(I-\lambda g))\left(  P_{C}(I-\lambda
f)\right)  \right)  ,
\end{align}
which means that $x^{\ast}\in\operatorname*{Fix}(P_{C}(I-\lambda f)))$ and
$x^{\ast}\in\operatorname*{Fix}(P_{Q}(I-\lambda g))$, and therefore by
(\ref{eq:fix-vip}) $x^{\ast}\in\Gamma.$ Finally, let $z\in\Gamma$, i.e., $z\in
SOL(C,f)\cap SOL(Q,g)$. Then $P_{Q}(z-\lambda g(z))=P_{C}(z-\lambda f(z))=z.$
Since the operators $P_{C}(I-\lambda f)$ and $P_{Q}(I-\lambda g)$ are
averaged, they are also nonexpansive. Thus%
\begin{align}
\left\Vert x^{k+1}-z\right\Vert ^{2}  &  =\left\Vert \left(  P_{C}(I-\lambda
f)\right)  (P_{Q}(I-\lambda g))(x^{k})-z\right\Vert ^{2}\nonumber\\
&  =\left\Vert \left(  P_{C}(I-\lambda f)\right)  (P_{Q}(I-\lambda
g))(x^{k})-P_{C}(I-\lambda f)(z)\right\Vert ^{2}\nonumber\\
&  \leq\left\Vert P_{Q}(x^{k}-\lambda g(x^{k}))-z\right\Vert ^{2}\nonumber\\
&  =\left\Vert P_{Q}(x^{k}-\lambda g(x^{k}))-P_{Q}(z-\lambda g(z))\right\Vert
^{2}\nonumber\\
&  \leq\left\Vert x^{k}-z\right\Vert ^{2}.
\end{align}
So%
\begin{equation}
\left\Vert x^{k+1}-z\right\Vert ^{2}\leq\left\Vert x^{k}-z\right\Vert ^{2},
\label{eq:Fejer}%
\end{equation}
which means that the sequence $\left\{  x^{k}\right\}  _{k=0}^{\infty}$ is
Fej\'{e}r-monotone with respect to $\Gamma$. Now, put%
\begin{equation}
u^{k}=P_{\Gamma}(x^{k}).
\end{equation}
Since the operators $P_{C}(I-\lambda f)$ and $P_{Q}(I-\lambda g)$ are
nonexpansive, it follows from (\ref{eq:fix-vip}) that the sets $SOL(C,f)\ $and
$SOL(Q,g)$ are nonempty, closed and convex (see \cite[Proposition 5.3, page
25]{Goebel}). In addition, since $\Gamma\neq\emptyset,$ each $u^{k}$ is well
defined. So, applying (\ref{eq:Proj}) with $D=\Gamma$ and $x=x^{k},$ we get%
\begin{equation}
\left\langle y-P_{\Gamma}\left(  x^{k}\right)  ,x^{k}-P_{\Gamma}\left(
x^{k}\right)  \right\rangle \leq0\text{ for all }k\geq0\text{ and }y\in\Gamma.
\end{equation}
Taking, in particular, $y=x^{\ast}\in\Gamma,$ we obtain%
\begin{equation}
\left\langle x^{\ast}-u^{k},x^{k}-u^{k}\right\rangle \leq0.
\end{equation}
By Lemma \ref{Lemma:Takahashi}, $\left\{  u^{k}\right\}  _{k=0}^{\infty}$
converges strongly to some $u^{\ast}\in\Gamma$. Therefore%
\begin{equation}
\left\langle x^{\ast}-u^{\ast},x^{\ast}-u^{\ast}\right\rangle \leq0
\end{equation}
and hence $u^{\ast}=x^{\ast}$, as asserted.
\end{proof}

\begin{remark}
1. The sequence $\left\{  y^{k}\right\}  _{k=0}^{\infty}$ also converges
weakly to $x^{\ast}\in\Gamma$.

2. Under the additional assumptions that $C$ and $Q$ are symmetric, and $f$
and $g$ are odd, that is, $f(-x)=-f(x)$ and $g(-x) = -g(x)$ for all
$x\in\mathcal{H}$, we get from
\cite[Corollary 2.1]{bbr} that any sequence $\left\{  x^{k}\right\}
_{k=0}^{\infty},$ generated by Algorithm \ref{Alg-Alternating}, converges
strongly to a point $x^{\ast}\in\Gamma$.

3. Strong convergence also occurs when either $C$ or $Q$ is compact.

4. According to \cite[Corollary 2.2]{bbr}, if $\Gamma=\emptyset,$ then
$\lim_{_{k\rightarrow\infty}}\left\Vert x^{k}\right\Vert =\infty$.

5.When $C$ and $Q$ are closed subspaces and $f=g=0$ in the two-set CSVIP
(\ref{eq:vip1}) and (\ref{eq:vip2}), we get von Neumann's original problem and
then Algorithm \ref{Alg-Alternating} is the classical alternating projections
method (\ref{eq:apm}).

6. In \cite[Algorithm 3.1]{cgrs10} we presented an algorithm which can be
applied to the solution of (\ref{eq:vip1}) and (\ref{eq:vip2}). The structure
of this algorithm is quite different from that of Algorithm
\ref{Alg-Alternating} in the sense that at each step there is the need to
calculate the projection of the current iterate onto the intersection of three
half-spaces. Although the latter calculation \textit{complicates the process,
the sequence generated in this way converges strongly to a solution.}
\end{remark}

Following \cite{Reich} and \cite{dr}, we now present two more algorithms for
solving the two-set CSVIP (\ref{eq:vip1}) and (\ref{eq:vip2}). Let
$\mathcal{H}$ be a real Hilbert space, and let $C$ and $Q$ be two nonempty,
closed and convex subsets of $\mathcal{H}$.

Recall the following two lemmata \cite[Lemmata 1.3 and 1.4]{Reich}.

\begin{lemma}
\label{Lemma:3.7}A convex combination of strongly nonexpansive operators is
also strongly nonexpansive.
\end{lemma}

\begin{lemma}
\label{Lemma:3.8}Let $T$ be a convex combination of the strongly nonexpansive
mappings $\left\{  T_{k}\mid1\leq k\leq m\right\}  $. If the set
$\; \cap\left\{\operatorname*{Fix}\left(  T_{k}\right)  \mid1\leq k\leq m\right\}  $ is not
empty, then it equals $\operatorname*{Fix}\left(  T\right)  $.
\end{lemma}

Now we can propose the following parallel algorithm.

\begin{algorithm}
\label{Alg-Alternating-1}$\left.  {}\right.  $

\textbf{Initialization:} Select an arbitrary starting point $x^{0}%
\in\mathcal{H}$, and let the numbers $w_{1}$ and $w_{2}$ be such that
$w_{1},w_{2}\geq0$ and $w_{1}+w_{2}=1$.

\textbf{Iterative step:} Given the current iterate $x^{k},$ compute%
\begin{equation}
x^{k+1}=w_{1}P_{C}\left(  x^{k}-\lambda f\left(  x^{k}\right)  \right)
+w_{2}P_{Q}\left(  x^{k}-\lambda g\left(  x^{k}\right)  \right)  .
\label{eq:convex-comb}%
\end{equation}

\end{algorithm}

\begin{theorem}
\label{Theorem:alternating-1} Let $\mathcal{H}$ be a real Hilbert space, and
let $C$ and $Q$ be two nonempty, closed and convex subsets of $\mathcal{H}$.
Assume that Conditions \ref{Condition:a}-\ref{Condition:c} hold. Then any
sequence $\left\{  x^{k}\right\}  _{k=0}^{\infty}$ generated by Algorithm
\ref{Alg-Alternating-1} converges weakly to a point $x^{\ast}\in\Gamma,$ and
furthermore,%
\begin{equation}
x^{\ast}=\lim_{k\rightarrow\infty}P_{\Gamma}(x^{k}).
\end{equation}

\end{theorem}

\begin{proof}
By Lemma \ref{Lemma:av}, the operators $P_{C}(I-\lambda f)$ and $P_{Q}%
(I-\lambda g)$ are averaged, hence strongly nonexpansive (see Remark
\ref{remark:av}). According to Lemma \ref{Lemma:3.7}, any convex combination
of strongly nonexpansive mappings is also strongly nonexpansive. So the
sequence $\left\{  x^{k}\right\}  _{k=0}^{\infty}$ generated by Algorithm
\ref{Alg-Alternating-1} is, in fact, an iteration of a strongly nonexpansive
operator and therefore the desired result is obtained by \cite{br} and Lemma
\ref{Lemma:3.8}.
\end{proof}

\begin{remark}
The convergence obtained in Theorem \ref{Alg-Alternating-1} is not strong in
general \cite{bmr}.
\end{remark}

Finally, we recall the following theorem \cite[Theorem 1]{dr}.

\begin{theorem}
Let $T_{1}:\mathcal{H}\rightarrow\mathcal{H}$ and $T_{2}:\mathcal{H}%
\rightarrow\mathcal{H}$ be two nonexpansive operators which satisfy Condition
(W), the fixed point sets of which have a nonempty intersection. Then any
unrestricted product from $T_{1}$ and $T_{2}$ converges weakly to a common
fixed point.
\end{theorem}

Since every averaged operator is strongly nonexpansive and therefore satisfies
Condition (W), we can apply the above theorem to obtain an algorithm for
solving the two-set CSVIP by using any unrestricted product from
$P_{C}(I-\lambda f)$ and $P_{Q}(I-\lambda g).$ Any such unrestricted product
converges weakly to a point in $\Gamma.$

\section{The general CSVIP\label{sec:General-CSVIP}}

In this section we extend our algorithm to two methods for solving the general
CSVIP with single-valued operators. Let $\mathcal{H}$ be a real Hilbert space.
Let there be given, for each $i=1,2,\ldots,N$, an operator $f_{i}%
:\mathcal{H}\rightarrow\mathcal{H}$ and a nonempty, closed and convex subset
$C_{i}\subset\mathcal{H}$, with $\bigcap_{i=1}^{N}C_{i}\neq\emptyset$. The
CSVIP is to find a point $x^{\ast}\in\bigcap_{i=1}^{N}C_{i}$ such that, for
each $i=1,2,\ldots,N,$%
\begin{equation}
\left\langle f_{i}(x^{\ast}),x-x^{\ast}\right\rangle \geq0\text{ for all }x\in
C_{i}\text{, }i=1,2,\ldots,N.
\end{equation}
Denote $\Psi:=\bigcap_{i=1}^{N}SOL(C_{i},f_{i})$.

\begin{algorithm}
\label{Alg-Alternating-2}$\left.  {}\right.  $

\textbf{Initialization:} Select an arbitrary starting point $x^{0}%
\in\mathcal{H}$.

\textbf{Iterative step:} Given the current iterate $x^{k},$ compute the
product%
\begin{equation}
x^{k+1}=\prod\limits_{i=1}^{N}\left(  P_{C_{i}}(I-\lambda f_{i})\right)
(x^{k}).
\end{equation}

\end{algorithm}

\begin{theorem}
\label{Theorem:1*} Let $\mathcal{H}$ be a real Hilbert space. For each
$i=1,2,\ldots,N$, let an operator $f_{i}:\mathcal{H}\rightarrow\mathcal{H}$
and a nonempty, closed and convex subset $C_{i}\subset\mathcal{H}$ be given.
Assume that $\bigcap_{i=1}^{N}C_{i}\neq\emptyset$, $\Psi\neq\emptyset$ and
that for $i=1,2,\ldots,N,$ $f_{i}$ is $\alpha_{i}$-ism. Set $\alpha:=\min
_{i}\{\alpha_{i}\}$ and take $\lambda\in(0,2\alpha)$. Then any sequence
$\left\{  x^{k}\right\}  _{k=0}^{\infty}$ generated by Algorithm
\ref{Alg-Alternating-2} converges weakly to a point $x^{\ast}\in\Psi,$ and
furthermore,%
\begin{equation}
x^{\ast}=\lim_{k\rightarrow\infty}P_{\Psi}(x^{k}).
\end{equation}

\end{theorem}

\begin{algorithm}
\label{Alg-Alternating-3}$\left.  {}\right.  $

\textbf{Initialization:} Select an arbitrary starting point $x^{0}%
\in\mathcal{H}$ and a nonnegative finite sequence $\left\{  w_{i}\right\}
_{i=1}^{N}$ such that $\sum\limits_{i=1}^{N}w_{i}=1.$

\textbf{Iterative step:} Given the current iterate $x^{k},$ compute%
\begin{equation}
x^{k+1}=\sum\limits_{i=1}^{N}w_{i}\left(  P_{C_{i}}(I-\lambda f_{i})\right)
(x^{k}).
\end{equation}

\end{algorithm}

\begin{theorem}
\label{Theorem:1**} Let $\mathcal{H}$ be a real Hilbert space. For each
$i=1,2,\ldots,N$, let an operator $f_{i}:\mathcal{H}\rightarrow\mathcal{H}$
and a nonempty, closed and convex subset $C_{i}\subset\mathcal{H}$ be given.
Assume that $\bigcap_{i=1}^{N}C_{i}\neq\emptyset$, $\Psi\neq\emptyset,$ and
that for $i=1,2,\ldots,N,$ $f_{i}$ is $\alpha_{i}$-ism. Set $\alpha:=\min
_{i}\{\alpha_{i}\}$ and take $\lambda\in(0,2\alpha)$. Then any sequence
$\left\{  x^{k}\right\}  _{k=0}^{\infty}$ generated by Algorithm
\ref{Alg-Alternating-3} converges weakly to a point $x^{\ast}\in\Psi,$ and
furthermore,%
\begin{equation}
x^{\ast}=\lim_{k\rightarrow\infty}P_{\Psi}(x^{k}).
\end{equation}

\end{theorem}

The proofs of Theorem \ref{Theorem:1*} and \ref{Theorem:1**} are analogous to
those of Theorems \ref{Theorem:1} and \ref{Theorem:alternating-1},
respectively, and therefore are omitted.\bigskip

\textbf{Acknowledgments}.

We gratefully acknowledge a referee's constructive comments that helped us to
improve the paper. This work was partially supported by United States-Israel
Binational Science Foundation (BSF) Grant number 200912, US Department of Army
Award number W81XWH-10-1-0170, Israel Science Foundation (ISF) Grant number
647/07, the Fund for the Promotion of Research at the Technion and by the
Technion VPR Fund.

\end{document}